\documentclass[12pt,reqno]{amsart}
\usepackage{amssymb,amsmath}
\usepackage{tikz}

\newtheorem{theorem}{Theorem}[section]
\newtheorem{proposition}[theorem]{Proposition}
\newtheorem{corollary}[theorem]{Corollary}
\newtheorem{lemma}[theorem]{Lemma}
\newtheorem{definition}[theorem]{Definition}
\newtheorem{remark}[theorem]{Remark}

\newtheorem{question}[theorem]{Question}

\newcommand\R{\mathbb R} 
\newcommand\CS{\mathcal S} 
\newcommand\PS{\mathcal P} 
\newcommand\M{\mathcal M} 
\newcommand\supp{{\rm supp}}
\newcommand\wt{{\rm wt}}
\newcommand\tA{{\widetilde{A}}}
\newcommand\ts{{\tilde{s}}}
\newcommand\hV{{\widehat{V}}}
\newcommand\hE{{\widehat{E}}}

\title{Dense Clusters in Hypergraphs}
\author{Yuly Billig}
\address{School of Mathematics and Statistics, Carleton University, Ottawa, Canada}
\email{billig@math.carleton.ca}
\subjclass[2020]{Primary 05C42, 05C65; Secondary 05C85, 90C35}
\begin{document}

\maketitle

\begin{abstract}
In this paper we solve the problem of finding in a given weighted hypergraph a subhypergraph with a maximum possible density. We introduce the notion of a support matrix and prove that the density of an optimal subhypergraph is equal to $\| A^T A \|$ for an optimal support matrix $A$. Alternatively,  the maximum density of a subhypergraph is equal to the solution of a minimax problem for column sums of support matrices. We introduce the spectral decomposition of a hypergraph and show that it is a significant refinement of the Dulmage-Mendelsohn decomposition.
Our theoretical results yield an efficient algorithm for finding the maximum density subhypergraph and more generally, the spectral decomposition for a given weighted hypergraph.
\end{abstract}

\section{Introduction}

Hypergraphs are generalizations of graphs where hyperedges are allowed to contain any number of vertices. In this paper we introduce new tools for clustering in hypergraphs. The methods we develop are applicable in data science. For example, a video streaming company may wish to cluster their customers by videos that customers have streamed. The data here may be represented as a hypergraph with videos being vertices and customers represented by hyperedges, where the support of a hyperedge is the set of videos streamed by the given customer. A cluster of customers with similar interests will correspond to a high density subhypergraph in the hypergraph of all customer records.

We define the density $\alpha(H)$ of a hypergraph $H$ as the ratio of the number of its hyperedges to the number of its vertices. In this paper we solve the following

\noindent
{\bf Densest Subgraph Problem:}
In a given hypergraph find a subhypergraph of maximum density.

Of course, the results obtained in this paper are applicable to graphs, being special cases of hypergraphs. For brevity, subhypergraphs will be called subgraphs in this paper.

For graphs, a solution of the Densest Subgraph Problem (DSP) was given by Goldberg \cite{Go} in 1984. Goldberg's algorithm finds the optimal solution of DSP by an iterative process, where each step consists of solving a maximum flow - minimum cut problem for a certain network associated with a graph, with capacities of the network being updated on each iteration.

Goldberg's algorithm was generalized to the setting of weighted hypergraphs by Hu, Wu, and Chan in \cite{HWC}. For a hypergraph $H$ with $n$ vertices and $m$ hyperedges the family of networks constructed in \cite{HWC} 
has the following parameters: the number of vertices is $N = n + m + 2$, the number of edges is 
$M = D + 2n$, where $D$ is the sum of the degrees of vertices in hypergraph $H$. 

Goldberg and Tarjan review the fastest known algorithms for max-flow problem in \cite{GT}. To date the best algorithm for solving max-flow problem is due to Goldberg and Rao \cite{GR} and has complexity $O(M \min(\sqrt{M}, N^{\frac{2}{3}}) \log(N^2 / M))$. Thus one iteration of a max-flow based algorithm for DSP can be implemented with complexity 
\break
$O(D \min(\sqrt{D}, (n+m)^{\frac{2}{3}}) \log((n+m)^2/D))$. The number of iterations required for this algorithm is $O(\log(n))$.

In the present paper we introduce a new iterative algorithm for DSP where each iteration has a much better complexity $O(D)$ and its space requirement is also $O(D)$. For hypergraphs  with integer weights we have a termination condition for the algorithm, which will guarantee that the constructed subgraph is optimal. Further research is required to establish the rate of convergence of our algorithm, however experiments suggest that for generic hypergraphs the algorithm terminates after $O(\log(n))$ iterations.

There are also fast greedy algorithms for DSP which do not produce the optimal solution, see e.g., \cite{C}.

We introduce the notion of a support matrix, which will be our main tool for solving this problem.
Let $V = \{ v_j \, | \, j=1,\ldots,n \}$, $E = \{ e_i \, | \, i=1,\ldots,m \}$ be vertices and hyperedges of $H$.
 
\begin{definition} 
For a hypergraph $H$ a support matrix $A$ is an $m \times n$ real matrix with
the following properties:

(1) $a_{ij} \geq 0$ for all $i = 1, \ldots, m$, $j=1,\ldots,n$.

(2) If vertex $v_j$ does not belong to hyperedge $e_i$ then $a_{ij} = 0$.

(3) $\sum\limits_{j=1}^n a_{ij} = 1$ \ for all $i = 1, \ldots,m$.
\end{definition}

When we cluster hyperedges, it is natural to look at the vertices they share. However, when assessing how close are two hyperedges to each other, the fact that they share a vertex that belongs to many other hyperedges is not as significant as the fact that they share a vertex of a low degree. Support matrices will allow us to automatically take this consideration into account.

For a given hypergraph $H$, the set of its support matrices forms a simplex $\CS(H)$. Our approach is to perform an optimization procedure in $\CS(H)$.

For $A \in \CS(H)$, matrix $A^T A$ is a symmetric matrix with non-negative real eigenvalues. The norm $\| A^T A \|$ is equal to its dominant eigenvalue. We relate a maximum density subgraph in $H$ to the norm of $A^T A$ for an optimal support matrix $A$.

\begin{theorem}
$$\max\limits_{H^\prime \subset H} \alpha(H^\prime) = \min\limits_{A \in \CS(H)} \| A^T A \|.$$
\end{theorem}

There is another optimization procedure in $\CS(H)$ which yields the same optimal support matrices. By definition, row sums in a support matrix are equal to 1. Let us consider column sums
$s_j(A) = \mathop\sum\limits_{i=1}^m a_{ij}$. Set $s_{\max}(A) = \max\limits_{j=1,\ldots,n} s_j(A)$.

\begin{theorem}
\label{smax}
$$\max\limits_{H^\prime \subset H} \alpha(H^\prime) = \min\limits_{A \in \CS(H)} s_{\max}(A).$$
\end{theorem}

It is easy to see that $H$ contains a unique maximal subgraph $H_{\alpha_1} = (V_{\alpha_1}, E_{\alpha_1})$ of maximum density $\alpha_1$. Consider the quotient hypergraph $H / H_{\alpha_1}$ (obtained from $H$ by removing all hyperedges of $H_{\alpha_1}$ and deleting all vertices of $H_{\alpha_1}$ from supports of the remaining hyperedges). Take $H_{\alpha_2} = (V_{\alpha_2}, E_{\alpha_2})$ to be the maximal subgraph in  $H / H_{\alpha_1}$ of maximum density $\alpha_2$. Iterating this process, we construct the spectral decomposition of $H$, which is a partitioning of the sets of vertices and hyperedges of $H$
$$V = V_{\alpha_1} \, \dot\cup \ldots \dot\cup \, V_{\alpha_k}, \quad
E = E_{\alpha_1} \, \dot\cup \ldots \dot\cup \, E_{\alpha_k},$$
yielding a chain of subgraphs in $H$:
$$H_{\alpha_1} \, \subset \, H_{\alpha_1} \cup H_{\alpha_2} 
\, \subset \, H_{\alpha_1} \cup H_{\alpha_2} \cup H_{\alpha_3} \, \subset \ldots \subset \, H.$$
 We show that there exists a support matrix $A$ with the property that for each vertex 
 $v_j \in V_{\alpha_r}$ we have $s_j (A) = \alpha_r$. 
 Thus the spectral decomposition of $H$ may be extracted from this optimal support matrix.
 Such an optimal support matrix necessarily has a block structure with $a_{ij} \neq 0$ only when 
 $v_j \in V_{\alpha_r}$,  $e_i \in E_{\alpha_r}$ for some $r$. 
 
 A hypergraph $H = (V, E)$ has a dual hypergraph $H^*$ with $E$ as the set of vertices of $H^*$ and $V$ as the set of hyperedges, with $e \in \supp_{H^*} (v)$ whenever $v \in \supp_H (e)$.
 We prove that the spectral decomposition of $H^*$ is the dual of the spectral decomposition of $H$.
 \begin{theorem}
 Let $\{ (V_{\alpha_1}, E_{\alpha_1}), (V_{\alpha_2}, E_{\alpha_2}),\ldots,(V_{\alpha_k}, E_{\alpha_k}) \}$ be the spectral decomposition of hypergraph $H$ with factors of densities
 $\alpha_1 > \alpha_2 > \ldots > \alpha_k > 0$. Then
 $\{ (E_{\alpha_k}, V_{\alpha_k}), \ldots, (E_{\alpha_2}, V_{\alpha_2}),(E_{\alpha_1}, V_{\alpha_1}) \}$ is the spectral decomposition of the dual hypergraph $H^*$ with factors of densities $\alpha_k^{-1} > \ldots >  \alpha_2^{-1} > \alpha_1^{-1} > 0$.
 \end{theorem}
 
 The spectral decomposition of a hypergraph is a significant refinement of its Dulmage-Mendelsohn decomposition \cite{DM} (see Section \ref{dulm} for the definition).
 
\begin{theorem}
Let $(H^+, H^0, H^-)$ be the Dulmage-Mendelsohn decomposition of hypergraph $H$. Then
$$H^+ = \mathop\cup\limits_{\alpha > 1} H_\alpha, \quad
H^0 = H_1, \quad 
H^- = \mathop\cup\limits_{\alpha < 1} H_\alpha.$$
\end{theorem}
 
Let $H$ be a hypergraph with density $\alpha$ such that no subgraph in $H$ has density exceeding $\alpha$. Then optimal support matrices have row sums equal to 1 and column sums equal to $\alpha$. Thus optimal support matrices may be viewed as the analogues of doubly stochastic matrices. It is well-known that doubly stochastic matrices form a simplex whose vertices are permutation matrices. Optimal support matrices also form a simplex.

\begin{question}
What are the vertices of the simplex of optimal support matrices?
\end{question}

Based on Theorem \ref{smax} we developed a simple efficient parallelizable algorithm for finding a maximum density subgraph (and more generally the spectral decomposition) for a given hypergraph. This algorithm is an iterative procedure applied to a support matrix $A$ which gradually minimizes $s_{\max}(A)$. The block structure in $A$ corresponding to the spectral decomposition will emerge even before the support matrix becomes optimal. Hence the maximum density subgraph may be extracted from a nearly optimal support matrix. At the same time a given support matrix provides an upper bound on densities of subgraphs. As a result, not only can we extract a high density subgraph from the support matrix, but the support matrix may be used to demonstrate that found subgraph is the best possible. Indeed, since possible values of densities are rational numbers with denominators bounded by the number of vertices, it is sufficient to construct a subgraph with density $\alpha$ and a support matrix $A$ such that $s_{\max}(A)$ is lower than the next rational number exceeding $\alpha$ with an admissible denominator. This algorithm allows us to find a subgraph which provably has the maximum density in a hypergraph with millions of hyperedges.

In this paper we work in the setting of weighted hypergraphs, where we assign positive weights to each vertex and each hyperedge. For simplicity of exposition, we presented a non-weighted version of our results in this Introduction.  

The structure of the paper is as follows. In Section \ref{hyper} we give our main definitions and state a key theorem about the maximum density subgraphs. In Section \ref{mproof} we give the proof of this key result. In Section \ref{spect} we introduce the spectral decomposition of a hypergraph. We present a relationship between spectral decompositions of a hypergraph and its dual in Section \ref{dualty} and show in Section \ref{dulm} that the spectral decomposition is a significant refinement of the Dulmage-Mendelsohn decomposition. We conclude the paper with the description of an efficient algorithm for finding the maximum density subgraph, and more generally, the spectral decomposition for a given hypergraph.
\subsection*{Acknowledgements}
This research is supported in part with a grant from the
Natural Sciences and Engineering Research Council of Canada. 

\section{Weighted hypergraphs and their support matrices}
\label{hyper}

A hypergraph is a triple $H = (V, E, \supp)$, where $V = \{ v_j \, | \, j=1, \ldots, n \}$  is the set 
of vertices, $E = \{ e_i \, | \, i=1, \ldots, m \}$ is the set of hyperedges, and the support function to the power set of $V$, $\supp: \, E \rightarrow \PS(V)$ indicates which vertices belong to a given hyperedge. In a weighted hypergraph we assign weights  $\wt(v_j) = w_j > 0, \wt(e_i) = u_i > 0$ to each vertex and hyperedge.

For subsets $V^\prime \subset V$, $E^\prime \subset E$ we define their weights as 
$$\wt(V^\prime) = \sum\limits_{v \in V^\prime} \wt(v), \ \ 
\wt(E^\prime) = \sum\limits_{e \in E^\prime} \wt(e).$$ 

 The dual hypergraph $H^*$ is a triple $(V^*, E^*, \supp^*)$ where $V^* = E$ and $E^* = V$. The support in $H^*$ of a hyperedge $e^* = v$ is defined as the set of hyperedges of $H$ that contain vertex $v$, that is, $v \in \supp (e)$ if and only if $e \in \supp^* (v)$. Weights $w_j$ of vertices in $H$ become weights of hyperedges in $H^*$ and vice versa.

 We define a {\it subhypergraph} $H^\prime \subset H$ as a pair of subsets $(V^\prime, E^\prime)$ where $V^\prime \subset V$, 
$E^\prime \subset E$ such that $\supp(e) \subset V^\prime$ for every hyperedge $e \in E^\prime$.
The weights of $H^\prime$ are inherited from $H$.

% In this paper we only deal with hypergraphs, so whenever we say ``subgraph" or ``edge", we mean subhypergraph and hyperedge.
 
  We define a {\it quotient hypergraph} $H^{\prime\prime}$ as a pair of subsets $(V^{\prime\prime}, E^{\prime\prime})$ such that for every $v \in V^{\prime\prime}$ and every 
  hyperedge $e \in E$ containing $v$, we have $e \in E^{\prime\prime}$. The hypergraph structure on $H^{\prime\prime}$ is defined by $\supp_{H^{\prime\prime}} (e) = \supp_H (e) \cap V^{\prime\prime}$.
  
  A quotient of a subgraph in $H$ will be called a factor of $H$.
  
  The following Lemma is elementary, and we omit its proof:
\begin{lemma}
(a) $H^\prime = (V^\prime, E^\prime)$ is a subgraph in $H$ if and only if $H^{\prime\prime} 
= (V \backslash V^\prime, E \backslash  E^\prime)$ is a quotient hypergraph of $H$.

\noindent
(b) $H^\prime = (V^\prime, E^\prime)$ is a subgraph in $H$ if and only if 
$(E^\prime, V^\prime)$ is a quotient hypergraph of $H^*$.
\end{lemma}
  
\begin{definition} 
The density of a weighted hypergraph is defined as 
$$\alpha(H) = \frac{ \sum\limits_{i=1}^m u_i}{\sum\limits_{j=1}^n w_j} = \frac{\wt(E)}{\wt(V)}.$$
\end{definition}

\begin{lemma} 
\label{density}
 Let $H$ be a weighted hypergraph.
 
 \noindent
(a) The dual hypergraph has density 
$\alpha(H^*) = \alpha(H)^{-1}.$

%\noindent
%(b) Let $H^\prime = (V^\prime, E^\prime)$ be a proper subgraph in $H$ and both $H$ and $H^\prime$ have density $\alpha$. Then the quotient hypergraph $H^{\prime\prime} = (V \backslash V^\prime, E \backslash E^\prime)$ also has density $\alpha$.  

\noindent 
(b) Let $H^\prime = (V^\prime, E^\prime)$ be a proper subgraph in $H$ with the quotient
$H^{\prime\prime} = (V^{\prime\prime} = V \backslash V^\prime, \, E^{\prime\prime} = E \backslash E^\prime)$. The density of $H$ is a weighted average of the densities of $H^\prime$ and $H^{\prime\prime}$:
$$\alpha(H) = \frac{\wt(V^\prime)}{\wt(V^\prime) + \wt(V^{\prime\prime})} \alpha(H^\prime)
+ \frac{\wt(V^{\prime\prime})}{\wt(V^\prime) + \wt(V^{\prime\prime})} \alpha(H^{\prime\prime}).$$

\noindent
(c) Suppose $H$ has no subgraphs $H^\prime \subset H$ with $\alpha(H^\prime) > \alpha(H)$.
Then its dual hypergraph has no subgraphs $\widetilde{H} \subset H^*$ with density 
$\alpha(\widetilde{H}) > \alpha(H)^{-1}.$
\end{lemma}

\begin{proof} 
The claims of part (a) and (b) are obvious. Let us prove part (c) by contradiction. 
% Suppose $\alpha(\widetilde{H}) > \alpha(H)^{-1}$.
% This means that 
% $$\frac{\wt(\widetilde{V})}{\wt(\widetilde{E})} > \frac{\wt(V)}{\wt(E)}.$$
% Then $\wt(\widetilde{V}) \wt(E) > \wt(\widetilde{E}) \wt(V)$, 
% which can be written as
% $$\left( \wt(V) - \wt(V \backslash \widetilde{V}) \right)  \wt(E) > 
% \left( \wt(E) - \wt(E \backslash \widetilde{E}) \right) \wt(V).$$
% After an obvious cancellation, we get
% $$\frac{\wt(E \backslash \widetilde{E})}{\wt(V \backslash \widetilde{V})} > 
% \frac{\wt(E)}{\wt(V)}.$$
% This gives us $\alpha(V \backslash \widetilde{V}, E \backslash \widetilde{E} ) > \alpha(H)$, which is a contradiction.
If a subgraph in $H^*$ has density greater than $\alpha(H)^{-1}$ then $H$ has a quotient hypergraph with density less than $\alpha(H)$. But by part (b) the complimentary subgraph in $H$ will have density exceeding $\alpha(H)$, which is a contradiction.
\end{proof}

 The goal of this paper is solving the following

\noindent
{\bf Densest Subgraph Problem:} In a given weighted hypergraph $H$ find a subgraph $H^\prime$ maximizing the density $\alpha(H^\prime)$.

For technical reasons from now on we will assume that each vertex belongs to some hyperedge, and  each hyperedge has a non-empty support.

 Our main tool for solving the above problem is the notion of a {\it support matrix} $A$.

\begin{definition} 
For a weighted hypergraph $H$ a support matrix $A$ is an $m \times n$ real matrix with
the following properties:

(1) $a_{ij} \geq 0$ for all $i = 1, \ldots, m$, $j=1,\ldots,n$.

(2) If vertex $v_j$ does not belong to hyperedge $e_i$ then $a_{ij} = 0$.

(3) $\sum\limits_{j=1}^n w_j a_{ij} = u_i$ \ for all $i = 1, \ldots,m$.
\end{definition}

 % A given hypergraph may have many support matrices. 
 We denote by $\CS(H)$ the set of all support matrices of $H$. We will be solving the problem of finding the optimal
subgraph in $H$ by running an optimization procedure in the set $\CS(H)$. 

 We point out that a support matrix tells us something about densities of (hidden) subgraphs in $H$. Given a support matrix $A$, let us denote by $s_j$ a sum of entries
in column $j$:
$$s_j = \sum\limits_{i=1}^m a_{ij}.$$
Set $s_{\max}(A)$ to be the maximum of the column sums:
$$s_{\max} (A) = \max\limits_{j=1,\ldots,n} s_j .$$

\begin{proposition}
\label{ineq}
 Let $H$ be a weighted hypergraph, and let $A$ be any of its support matrices. Then the density of any subgraph $H^\prime \subset H$ does not exceed $s_{\max} (A)$:
$$ \alpha(H^\prime) \leq s_{\max} (A) .$$
\end{proposition}
\begin{proof} 
Let $H^\prime = (V^\prime, E^\prime)$ be a subgraph in $H$. Let us compute the weighted sum of the entries of $A$ in rows corresponding to $E^\prime$:
$$\sum\limits_{i \in E^\prime} \sum_{j=1}^n w_j a_{ij} = \sum\limits_{i \in E^\prime} \sum_{j \in V^\prime} w_j a_{ij} = \sum\limits_{i \in E^\prime} u_i .$$

Let us also compute a weighted sum in columns corresponding to $V^\prime$:
$$\sum\limits_{j \in V^\prime} w_j s_j  = \sum\limits_{j \in V^\prime} \sum\limits_{i=1}^m w_j a_{ij} 
\geq \sum\limits_{j \in V^\prime} \sum\limits_{i \in E^\prime} w_j a_{ij} = \sum\limits_{i \in E^\prime} u_i .$$ 
Then we get
$$s_{\max} (A) = \frac{\sum\limits_{j\in V^\prime} w_j s_{\max} (A)}{\sum\limits_{j\in V^\prime} w_j}
\geq \frac{\sum\limits_{j\in V^\prime} w_j s_j}{\sum\limits_{j\in V^\prime} w_j} 
\geq \frac{\sum\limits_{i \in E^\prime} u_i}{\sum\limits_{j\in V^\prime} w_j} = \alpha(H^\prime).$$
\end{proof}

Let us state our main result:

\begin{theorem}
\label{main}
Let $H$ be a weighted hypergraph. Then 
$$\max\limits_{H^\prime \subset H} \alpha(H^\prime) = \min\limits_{A \in \CS(H)} s_{\max} (A).$$
\end{theorem}

\section{Proof of the Main Theorem.}
\label{mproof}

It follows from Proposition \ref{ineq} that
$$\max\limits_{H^\prime \subset H} \alpha(H^\prime) \leq \min\limits_{A \in \CS(H)} s_{\max} (A).$$ 
The opposite inequality will follow from

\begin{proposition}
\label{existence}
 (a) Let $H$ be a weighted hypergraph with density $\alpha$ such that no subgraph in $H$ has density exceeding $\alpha$.
Then there exists a support matrix $A$ for $H$ such that every column sum is equal to $\alpha$:
$$s_j (A) = \sum\limits_{i=1}^m a_{ij} = \alpha {\rm \ \ for \ all \ } j=1,\ldots, n.$$

\noindent
(b) Let $H$ be a weighted hypergraph such that no subgraph in $H$ (including $H$ itself) has density exceeding $\alpha$. Then there exists a support matrix $A$ for $H$
such that 
$$s_j (A) = \sum\limits_{i=1}^m a_{ij} \leq \alpha {\rm \ \ for \ all \ } j=1,\ldots, n.$$
\end{proposition}

\begin{remark}  If we succeed to construct an optimal support matrix $A$ for $H$ minimizing $s_{\max}(A)$ then we will be able to identify a subgraph $H^\prime = (V^\prime, E^\prime)$ in $H$ of maximum density constructing the set $V^\prime$ by taking the columns of $A$ satisfying $s_j = s_{\max} (A)$, and defining the set $E^\prime$ by
$$E^\prime = \{ e_i \in E \, | \, \exists v_j \in V^\prime \ \ a_{ij} \neq 0 \}.$$
Thus the optimal subgraph in $H$ may be extracted from an optimal support matrix, and even from an approximation to an optimal support matrix. See Section \ref{algo} for more details.
\end{remark}

Hu, Wu, and Chan showed (\cite{HWC}, Lemma 2.2, see also \cite{C}, Lemma 1)  that a solution of DSP may be obtained from the solution of the following linear programming problem:
\begin{lemma}
Let $H$ be a weighted hypergraph. Consider a linear optimization problem on variables
$p_i \geq 0, \, q_j \geq 0, \ i=1,\ldots,m, \ j=1,\ldots,n$ 
subject to constraints $p_i \leq q_j$ whenever vertex $v_j$ belongs to hyperedge $e_i$ 
and $\sum\limits_{j=1}^n w_j q_j = 1$, with the objective function 
$\sum\limits_{i=1}^m u_i p_i \rightarrow \max$. Then under these constraints
$$\max \sum\limits_{i=1}^m u_i p_i = \max\limits_{H^\prime \subset H} \alpha(H^\prime). $$
\end{lemma}

\begin{corollary}
\label{constraints}
 Let $H$ be a weighted hypergraph such that no subgraph in $H$ (including $H$ itself) has density exceeding $\alpha$.
Let $p_i \geq 0, \ i=1,\ldots,m, \ q_j \geq 0, \ j=1,\ldots,n$ satisfy $p_i \leq q_j$ whenever vertex $v_j$ belongs to hyperedge $e_i$.
Then
$$\sum\limits_{i=1}^m u_i p_i \leq \alpha \sum\limits_{j=1}^n w_j q_j .$$
\end{corollary}

 Now we are ready to proceed with the proof of Proposition \ref{existence}. We will be using the methods of linear programming.
For part (a) we shall view the task of finding a required support matrix $A$ as a linear programming problem
on the set of unknowns $\left\{ a_{ij} \, | \, v_j \in \supp(e_i) \right\}$
with constraints $a_{ij} \geq 0$,
$$ [e_i]: \ \ \sum\limits_{j = 1}^n w_j a_{ij} = u_i, \ \ i=1,\ldots, m,$$
$$[v_j]: \ \ \sum\limits_{i=1}^m a_{ij} = \alpha, \ \ j=1,\ldots,n.$$
If a desired support matrix $A$ does not exist then this linear programming problem is infeasible.

Farkas' Lemma (\cite{Vdb}, Section 10.4) on infeasible linear programming problems states that in this case there is a linear
combination of constraint equations 
$$\sum\limits_{i=1}^m f_i [e_i] + \sum\limits_{j=1}^n g_j [v_j]$$
such that the coefficients of all $a_{ij}$ in it are non-negative, but the right-hand-side is negative. This yields inequalities
\begin{equation}
\label{first}
w_j f_i + g_j \geq 0 \ \ {\rm whenever \ vertex \ } v_j {\rm \ belongs \ to \ hyperedge \ } e_i,
\end{equation}
and 
\begin{equation}
\label{second}
\sum\limits_{i=1}^m f_i u_i + \alpha \sum\limits_{j=1}^n g_j < 0.
\end{equation}

We split $V$ and $E$ into subsets according to the signs of $f_i$ and $g_j$:
$$V^\prime = \left\{ j \in V \, | \, g_j \geq 0 \right\}, \ \ 
V^{\prime\prime} = \left\{ j \in V \, | \, g_j < 0 \right\}.$$
$$E^\prime = \left\{ i \in E \, | \, f_i < 0 \right\}, \ \ 
E^{\prime\prime} = \left\{ i \in E \, | \, f_i \geq 0 \right\},$$

Condition (\ref{first}) implies that for every $i \in E^{\prime}$, $j\in V^{\prime\prime}$ 
vertex $v_j$ does not belong to the hyperedge $e_i$,
while (\ref{first}) holds trivially for $i \in E^{\prime\prime}$, $j\in V^\prime$.

Then we conclude that $(V^\prime, E^{\prime})$ is a subgraph in $H$, while $(E^{\prime\prime}, V^{\prime\prime})$ is a subgraph in $H^*$.

Set $p_i = |f_i|$ and $q_j = \frac{|g_j|}{w_j}$. Then assuming that vertex $v_j$ belongs to hyperedge $e_i$, condition (\ref{first}) becomes
$$q_j \geq p_i \geq 0 \ \ {\rm if} \ i \in E^{\prime}, j \in V^\prime,$$
$$p_i \geq q_j \geq 0 \ \ {\rm if} \ i \in E^{\prime\prime}, j \in V^{\prime\prime},$$
while condition (\ref{second}) becomes
\begin{equation}
\label{third}
\left( \alpha \sum\limits_{j\in V^\prime} q_j w_j - \sum\limits_{i\in E^{\prime}} u_i p_i \right) +
\left( \sum\limits_{i\in E^{\prime\prime}} u_i p_i - \alpha \sum\limits_{j\in V^{\prime\prime}} q_j w_j \right) < 0.
\end{equation}

 Since the density of $(V^\prime, E^{\prime})$ and of its subgraphs is at most $\alpha$, we have by Corollary \ref{constraints}
$$ \sum\limits_{i\in E^{\prime}} u_i p_i \leq \alpha \sum\limits_{j\in V^\prime} q_j w_j.$$

 By Lemma \ref{density}(c), $H^*$ has no subgraphs of density exceeding $\alpha^{-1}$. Applying Corollary \ref{constraints} to $(E^{\prime\prime}, V^{\prime\prime}) \subset H^*$,
we get 
$$\sum\limits_{j\in V^{\prime\prime}} q_j w_j \leq \alpha^{-1} \sum\limits_{i\in E^{\prime\prime}} u_i p_i.$$
Combining the last two inequalities, we get a contradiction to (\ref{third}).

The proof of part (b) of Proposition \ref{existence} goes along the same lines as for part (a). The linear programming problem will be modified to have inequalities in constraints $[v_j]$ instead of the equalities, while constraints $[e_i]$ will stay as equalities. As a result we will get an additional restriction that $g_j \geq 0$, and the set $V^{\prime\prime}$ will become empty. Applying Corollary \ref{constraints}  to subgraph 
$(V^\prime, E^{\prime}) \subset H$ we will get a contradiction to (\ref{second}), which implies that the linear programming problem is 
feasible and the required support matrix $A$ exists. This completes the proof of Proposition \ref{existence} and our Main Theorem \ref{main} follows.

\begin{remark}
\label{block}
Let $H^\prime = (V^\prime, E^\prime)$ be a subgraph in $H$ of density $\alpha$ and let $A$ be a support matrix for $H$ with $s_j = \alpha$ for every vertex $v_j \in V^\prime$. Then it follows from the proof of Proposition \ref{ineq} that $A$ has a block decomposition with blocks $(V^\prime, E^\prime)$ and $(V \backslash V^\prime, E  \backslash E^\prime)$, that is,  we must have $a_{ij} = 0$ when $v_j \in V^\prime$ and $e_i \not\in E^\prime$ or when $v_j \not\in  V^\prime$ and $e_i \in E^\prime$. 
\end{remark}

\section{Spectral decomposition of a hypergraph}
\label{spect}

In this section we will present an eigenvalue interpretation of optimal support matrices and introduce the spectral decomposition of a hypergraph.

We recall that the norm of a real $n \times n$ matrix $B$ is defined as 
$$\| B \| = \sup\limits_{x \in \R^n \backslash \{ 0 \} } \frac{\| Bx\|}{\| x \|}.$$
This is equal to the maximum norm of a complex eigenvalue of $B$. If $B$ is symmetric and non-negative-definite then all eigenvalues of $B$ are real and non-negative, so the norm of $B$ is equal to its largest eigenvalue. In this case, the norm can also be given as
$$\| B \| = \sup\limits_{x \in \R^n \backslash \{ 0 \} } \frac{(x, Bx)}{(x, x)}.$$
Furthermore, if $B$ is factored as $B = Q^T Q$ for some $m \times n$ real matrix $Q$ then
$$\| B \| = \sup\limits_{x \in \R^n \backslash \{ 0 \} } \frac{(Qx, Qx)}{(x, x)}.$$

Consider diagonal $n \times n$ and $m \times m$ matrices of weights $W$ and $U$, where $j$-th entry on the diagonal of $W$ is $w_j$, and $i$-th entry on the diagonal of $U$ is $u_i$.
Let $W^{\frac {1}{2}}$ and $U^{\frac {1}{2}}$ be the diagonal matrices with
the diagonal entries $\sqrt{w_j}$ and $\sqrt{u_i}$ respectively.

For a support matrix $A$ define
$$\tA = U^{-\frac {1}{2}} A W^{\frac {1}{2}}.$$

\begin{theorem}
\label{eigen}
Let $H$ be a weighted hypergraph. Then 
$$\max\limits_{H^\prime \subset H} \alpha(H^\prime) = \min\limits_{A \in \CS(H)} \| \tA^T \tA \|.$$
\end{theorem}

Before we prove this theorem let us introduce the spectral decomposition of a hypergraph. We begin with an elementary Lemma:
\begin{lemma}
\label{boolean}
Let $H$ be a weighted hypergraph with no subgraphs of density exceeding $\alpha$. Let $H_1 = (V_1, E_1)$ and $H_2 = (V_2, E_2)$ be two subgraphs in $H$ of density $\alpha$. Then 
$H_1 \cup H_2 = (V_1 \cup V_2, E_1 \cup E_2)$ and $H_1 \cap H_2 = (V_1 \cap V_2, E_1 \cap E_2)$ also have density $\alpha$.
\end{lemma}
\begin{proof}
We have 
\begin{align*}
\wt(E_1 \cup E_2) &= \wt(E_1) + \wt(E_2) - \wt(E_1 \cap E_2) \\
&\geq \alpha \wt(V_1) + \alpha \wt(V_2) - \alpha \wt( V_1 \cap V_2) = \alpha \wt(V_1 \cup V_2).
\end{align*}
Since no subgraph in $H$ has density exceeding $\alpha$, the above inequality is in fact an equality, which yields the claims of the Lemma.
\end{proof}
\begin{corollary}
Let $\alpha = \max\limits_{H^\prime \subset H} \alpha(H^\prime)$. Then there exists a unique maximal subgraph in $H$ of density $\alpha$.
\end{corollary}
\begin{definition}
The spectral decomposition of a weighted hypergraph $H$ is a partitioning 
$$V = V_{\alpha_1} \ \dot\cup \ V_{\alpha_2} \ \dot\cup \ldots \dot\cup \ V_{\alpha_k}, \ \ 
E = E_{\alpha_1} \ \dot\cup \ E_{\alpha_2} \ \dot\cup \ldots \dot\cup \ E_{\alpha_k}, $$
defined inductively as $(V_{\alpha_t}, E_{\alpha_t})$ being the maximal subgraph of maximum density $\alpha_t$ in the quotient graph $(V \backslash \mathop\cup\limits_{r=1}^{t-1} V_{\alpha_r},
E \backslash \mathop\cup\limits_{r=1}^{t-1} E_{\alpha_r})$ for $t = 1, \ldots, k$.
\end{definition}
Clearly, for the spectral decomposition, $\alpha_1 > \alpha_2 > \ldots > \alpha_k > 0$.

The following theorem is an immediate corollary of Theorem \ref{main} and Remark \ref{block}:

\begin{theorem}
\label{colsum}
Let $(V = \mathop\cup\limits_{r=1}^k V_{\alpha_r}, E = \mathop\cup\limits_{r=1}^k E_{\alpha_r})$ be the spectral decomposition of a weighted hypergraph $H$. Then there exists 
a support matrix $A$ for $H$ such that for any vertex $v_j \in V_{\alpha_r}$ the column sum
$s_j (A)$ is equal to $\alpha_r$. Such a matrix will have a block decomposition with blocks 
corresponding to the factors of the spectral decomposition 
$(V_{\alpha_1}, E_{\alpha_1}), (V_{\alpha_2}, E_{\alpha_2}), \ldots,
(V_{\alpha_k}, E_{\alpha_k})$. 
\end{theorem}
Let us now give a proof of Theorem \ref{eigen}.
We begin by showing that for any subgraph $H^\prime = (V^\prime, E^\prime) \subset H$ and any support matrix $A \in \CS(H)$
 inequality $\alpha(H^\prime) \leq \| \tA^T \tA \|$ holds. Construct vector $x \in \R^n$
 with 
 $$x_j = \left\{ 
 \begin{matrix}
 \sqrt{w_j}, \ \text{if \ } v_j \in V^\prime, \\
 0, {\hskip 0.6cm} \text{otherwise.}
 \end{matrix} \right.$$
 Then $(x, x) = \wt(V^\prime)$. Let us compute vector $U^{-\frac {1}{2}} A W^{\frac {1}{2}} x$.
 We have
 $$(W^{\frac {1}{2}} x)_j = \left\{ 
 \begin{matrix}
 w_j, \ \text{if \ } v_j \in V^\prime, \\
 0, {\hskip 0.4cm} \text{otherwise.}
 \end{matrix} \right.$$
 If $e_i \in E^\prime$ then $\supp(e_i) \subset V^\prime$. Recalling the defining properties of a support matrix, we get
$$(A W^{\frac {1}{2}} x)_i = \left\{ 
 \begin{matrix}
 u_i, {\hskip 2cm} \text{if } e_i \in E^\prime,  \\
 \text{non-negative}, \ \text{otherwise,} 
 \end{matrix} \right.$$
 and
$$(U^{-\frac {1}{2}} A W^{\frac {1}{2}} x)_i = \left\{ 
 \begin{matrix}
 \sqrt{u_i}, {\hskip 1.7cm} \text{if } e_i \in E^\prime, \\
 \text{non-negative}, \  \text{otherwise,}
 \end{matrix} \right.$$
 From this we see that $(\tA x, \tA x) \geq \wt(E^\prime)$. Thus
 $$\alpha(H^\prime) \leq \frac{(\tA x, \tA x)}{(x, x)} \leq \| \tA^T \tA \| .$$
Since $A$ here is an arbitrary support matrix, and $H^\prime$ is an arbitrary subgraph, we get that
$$\max\limits_{H^\prime \subset H} \alpha(H^\prime) \leq \min\limits_{A \in \CS(H)} \| \tA^T \tA \|.$$
Let us show that in fact we have an equality by showing that for a maximum density subgraph there exists a support matrix on which the equality holds.

Let $H_{\alpha_1} = (V_{\alpha_1}, E_{\alpha_1})$ be the maximal subgraph in $H$ of maximum density $\alpha_1$. Consider the spectral decomposition of $H$. For each factor 
$(V_{\alpha_r}, E_{\alpha_r})$ we can construct a support matrix with all column sums equal to $\alpha_r$. Construct a support matrix $A$ for $H$ by arranging support matrices for each factor as blocks, and placing all zeros outside of these blocks. Then for $v_j \in V_{\alpha_r}$ the column sum $s_j (A)$ equals $\alpha_r$. We have constructed above eigenvectors 
$x^{(1)}, \ldots, x^{(k)}$ for $\tA^T \tA$ with eigenvalues $\alpha_1, \ldots, \alpha_k$. 
Each vector $x^{(r)}$ has positive entries in positions corresponding to $V_{\alpha_r}$ and zeros elsewhere.

Matrix $\tA^T \tA$ is symmetric and non-negative-definite. Thus its eigenvalues are real and non-negative. Its eigenvectors corresponding to distinct eigenvalues will be mutually orthogonal. Also $\tA^T \tA$ has non-negative entries. By Perron-Frobenius Theorem, $\tA^T \tA$ has a dominant eigenvector with all non-negative components (\cite{G}, Theorem XIII.3). This vector can not have a zero dot product with all of the vectors $x^{(1)}, \ldots, x^{(k)}$. Thus the dominant eigenvalue is one of $\alpha_1, \ldots, \alpha_k$, which means that the dominant eigenvalue of $A$ is precisely $\alpha_1$, the maximum density of a subgraph in $H$. This implies the claim of Theorem \ref{eigen}. 

Block decomposition of an optimal support matrix given by Theorem \ref{colsum} may be refined even further. Consider one of the factors $H_\alpha = (V_\alpha, E_\alpha)$ in the spectral decomposition of $H$. By definition, $H_\alpha$ has no subgraphs of density exceeding $\alpha$. Then by Lemma \ref{boolean}, the set of subgraphs of density $\alpha$ in $H_\alpha$ forms a distributive lattice. Consider the factors in this lattice (subgraphs of density $\alpha$ or quotients of two nested subgraphs of density  $\alpha$). We call such a factor simple if it has no proper subgraphs of density $\alpha$. We get partitionings of $V_\alpha$ and $E_\alpha$ corresponding to the simple factors of this lattice and an optimal support matrix will have a block structure with blocks corresponding to these simple factors.

\section{Duality}
\label{dualty}

In this section we would like to study the relationship between the spectral decompositions of a hypergraph $H$ and its dual $H^*$.

\begin{theorem}
\label{dual}
Let $\left\{ (V_1, E_1), \ldots, (V_k, E_k) \right\}$ be the spectral decomposition with densities $\alpha_1 > \ldots > \alpha_k > 0$ for a weighted hypergraph $H$.
Then $\left\{ (E_k, V_k), \ldots, (E_1, V_1) \right\}$ is the spectral decomposition with densities $\alpha_k^{-1} > \ldots > \alpha_1^{-1} > 0$ for the dual weighted hypergraph $H^*$.
\end{theorem}

To prove this theorem we will first show that an optimal support matrix determines the spectral decomposition of a hypergraph.

\begin{proposition}
\label{recover}
Suppose $\alpha_1 > \alpha_2 > \ldots > \alpha_k$. Let
\begin{equation}
\label{sdec}
V = V_{\alpha_1} \ \dot\cup \ V_{\alpha_2} \ \dot\cup \ldots \dot\cup \ V_{\alpha_k}, \ \ 
E = E_{\alpha_1} \ \dot\cup \ E_{\alpha_2} \ \dot\cup \ldots \dot\cup \ E_{\alpha_k} 
\end{equation}
be partitionings of the sets of vertices and hyperedges of a weighted hypergraph $H$ such that for every $t = 1, \ldots, k$,
$(\mathop\cup\limits_{r=1}^t V_{\alpha_r}, \mathop\cup\limits_{r=1}^t E_{\alpha_r})$ is a subgraph in $H$
and $\wt(E_r) / \wt(V_r) = \alpha_r$. Let $A$ be a support matrix for $H$ such that for every 
vertex $v_j \in V_{\alpha_r}$ the column sum $s_j(A)$ is equal to $\alpha_r$. Then (\ref{sdec})
is the spectral decomposition of $H$.
\end{proposition}
\begin{proof}
We prove the Proposition by induction on $k$. First, let us show that $(V_{\alpha_1}, E_{\alpha_1})$ is the maximal subgraph of maximum density in $H$. The density of this subgraph is $\alpha_1$ and it follows from Proposition \ref{ineq} that $\alpha_1$ is the maximum density of a subgraph in $H$. By the same Proposition, the quotient hypergraph $H^{\prime\prime} = (V \backslash V_{\alpha_1}, E \backslash E_{\alpha_1})$ has no subgraphs of density exceeding $\alpha_2$. This implies that 
$(V_{\alpha_1}, E_{\alpha_1})$ is the maximal subgraph of maximum density. By induction assumption, $(\mathop\cup\limits_{r=2}^k V_{\alpha_r}, \mathop\cup\limits_{r=2}^k E_{\alpha_r})$ is the spectral decomposition of the quotient graph $H^{\prime\prime}$.
The claim of the Proposition \ref{recover} now follows.
\end{proof}

To prove the duality Theorem \ref{dual}, we show how to construct an optimal support matrix for the dual hypergraph $H^*$ from an optimal support matrix $A$ for $H$. Let $A$ be an optimal support matrix given by Theorem \ref{colsum}.
% Since $A$ has a block structure associated with the spectral decomposition of $H$ (see Remark \ref{block}), for each column $j$ with $v_j \in V_{\alpha_r}$ the column sum $s_j (A)$ is equal to $\alpha_r$.

Let us construct a support matrix for the dual hypergraph $H^*$. Let $B$ be an $n \times m$ matrix with entries
$$b_{ji} = \alpha_r^{-1} u_i^{-1} a_{ij} w_j,$$
where $v_j \in V_{\alpha_r}$ in the spectral decomposition of $H$. Clearly the entries of $B$ are non-negative, and $a_{ij} = 0$ whenever vertex $e_i$ does not belong to hyperedge $v_j$ in $H^*$, and hence $b_{ji} = 0$. Let us evaluate weighted row sums in $B$:
$$\sum_{i=1}^m u_i b_{ji} = \alpha_r^{-1} w_j \sum_{i=1}^m a_{ij} = w_j.$$
Thus matrix $B$ is a support matrix for $H^*$. Let us evaluate the column sums for $B$. Since
$B$ has a block structure, for $i$ with $e_i \in E_{\alpha_r}$ all non-zero entries $b_{ji}$ in column $i$ of $B$ occur only in rows $j$ with $v_j \in V_{\alpha_r}$. Suppose $e_i \in E_{\alpha_r}$. Then 
$$\sum_{j=1}^n b_{ji} = \alpha_r^{-1} u_i^{-1} \sum_{j=1}^n w_j a_{ij} = \alpha_r^{-1}.$$

We can now see that matrix $B$ satisfies the conditions of Proposition \ref{recover} for $H^*$.
We have $\alpha_k^{-1} > \alpha_{k-1}^{-1} > \ldots > \alpha_1^{-1}$. Quotient graphs for $H$ correspond to subgraphs in $H^*$. Thus for each $t$, $(\mathop\cup\limits_{r=t}^k E_{\alpha_r}, \mathop\cup\limits_{r=t}^k V_{\alpha_r})$ is a subgraph in $H^*$. 
We have $\wt(V_{\alpha_r})/\wt(E_{\alpha_r}) = \alpha_r^{-1}$ and for every $e_i \in E_{\alpha_r}$ the column sum in $B$ equals $\alpha_r^{-1}$. Thus by Proposition \ref{recover}, the dual of the spectral decomposition of $H$ is the spectral decomposition of $H^*$.

\section{Refinement of the Dulmage-Mendelsohn decomposition}
\label{dulm}

In this section we are going to show that the spectral decomposition of a hypergraph is a refinement of the Dulmage-Mendelsohn decomposition. For this section we will assume that all weights $u_i$ and $w_j$ are equal to 1.

Let us recall the construction of the Dulmage-Mendelsohn decomposition of a hypergraph $H$ \cite{DM}.
This construction is based on the notion of the minimal exterior cover of $H$.

\begin{definition} A pair of subsets $(\hV, \hE)$, $\hV \subset V$, $\hE \subset E$, is called an exterior cover of $H$ if for any pair $v\in V$, $e \in E$ with $v \in \supp(e)$, either
$v \in \hV$ or $e \in \hE$. An exterior cover $(\hV, \hE)$ is called minimal if it minimizes
$|\hV| + |\hE|$ over all exterior covers of $H$.
\end{definition}

Let $\M$ be the set of all minimal exterior covers of $H$. Set
\begin{align*}
V^+ = \mathop\cap\limits_{(\hV, \hE) \in \M} \hV, \quad
V^- = \mathop\cap\limits_{(\hV, \hE) \in \M} V \backslash \hV, \quad
V^0 = V \backslash ( V^+ \cup V^-), \\
E^+ = \mathop\cap\limits_{(\hV, \hE) \in \M} E \backslash \hE, \quad
E^- = \mathop\cap\limits_{(\hV, \hE) \in \M} \hE, \quad
E^0 = E \backslash ( E^+ \cup E^-).
\end{align*}

Dulmage and Mendelsohn proved that $(V^+, E^+)$ and $(V^+ \cup V^0, E^+ \cup E^0)$ are subgraphs in $H$ \cite{DM}.

\begin{theorem}
Let $(V^+ \, \dot\cup \, V^0 \, \dot\cup \, V^-, E^+ \, \dot\cup \, E^0 \, \dot\cup \, E^-)$ be the 
Dulmage-Mendelsohn decomposition of hypergraph $H$, and let $\{(V_{\alpha_r}, E_{\alpha_r}) \, | \, r=1,\ldots,k \}$ be the spectral decomposition of $H$. Then 
\begin{align*}
V^+ = \mathop\cup\limits_{\alpha > 1} V_\alpha, \quad V^0 = V_1, \quad V^- = \mathop\cup\limits_{\alpha < 1} V_\alpha, \\
E^+ = \mathop\cup\limits_{\alpha > 1} E_\alpha, \quad E^0 = E_1, \quad E^- = \mathop\cup\limits_{\alpha < 1} E_\alpha.
\end{align*}
\end{theorem}
\begin{proof}
First of all we point out that both 
\begin{equation}
\label{ext}
( \mathop\cup\limits_{\alpha > 1} V_\alpha, \mathop\cup\limits_{\alpha \leq 1} E_\alpha) \text{ \ and \ } ( \mathop\cup\limits_{\alpha \geq 1} V_\alpha, \mathop\cup\limits_{\alpha < 1} E_\alpha)
\end{equation}
 are exterior covers for $H$. 

Let $(\hV, \hE)$ be a minimal exterior cover for $H$. Consider $H_\alpha = (V_\alpha, E_\alpha)$ as a subgraph of a quotient hypergraph of $H$. Set $\hV_\alpha = \hV \cap V_\alpha$, $\hE_\alpha = \hE \cap E_\alpha$. Then $(\hV_\alpha, \hE_\alpha)$ is an exterior cover of $H_\alpha$. We would like to show that for $\alpha > 1$ we must have $\hV_\alpha = V_\alpha$,
$\hE_\alpha = \varnothing$ and for $\alpha < 1$ we must have $\hV_\alpha = \varnothing$,
$\hE_\alpha = E_\alpha$.

Fix $\alpha \geq 1$. Let $\hE_\alpha^\prime = E_\alpha \backslash \hE_\alpha$. By the definition of the exterior cover, hyperedges in $\hE_\alpha^\prime$ must have support in $\hV_\alpha$. Thus
$(\hV_\alpha, \hE_\alpha^\prime)$ is a subgraph in $H_\alpha$. Since the densities of subgraphs in $H_\alpha$ do not exceed $\alpha$, we have $|\hE_\alpha^\prime| \leq \alpha |\hV_\alpha|$.
Then 
$$|V_\alpha| = \frac{1}{\alpha} |E_\alpha|
= \frac{1}{\alpha} (|\hE_\alpha| + |\hE_\alpha^\prime|) 
\leq \frac{1}{\alpha} |\hE_\alpha| + |\hV_\alpha| 
\leq |\hE_\alpha| + |\hV_\alpha| .$$
By duality we get that for $\alpha \leq 1$ we have $|E_\alpha| \leq |\hE_\alpha| + |\hV_\alpha|$. Taking the sum over all $\alpha$, we get that $|\hV| + |\hE|$ is greater or equal to the sizes of exterior covers (\ref{ext}). Hence both exterior covers in (\ref{ext}) are minimal.

From the above computation we see that for any minimal exterior cover $(\hV, \hE)$ of $H$ we  must have $\hE_\alpha = \varnothing$, $\hV_\alpha = V_\alpha$ for $\alpha > 1$ and, by duality,
$\hV_\alpha  = \varnothing$, $\hE_\alpha = E_\alpha$ for $\alpha < 1$. Recalling the definition of the Dulmage-Mendelsohn decomposition, we obtain the claim of the Theorem.
\end{proof}

\section{Algorithm for finding a subgraph of the maximum density.}
\label{algo}

 Our approach is to find a support matrix $A$ in $\CS(H)$ with a minimum value of $s_{\max}(A)$ (or rather a support matrix $A$ with $s_{\max}(A)$ close to the optimal value) and then extract the optimal subgraph from the support matrix. We will run an iterative process calculating successive approximations leading to the optimal support matrix, decreasing the value of $s_{\max}(A)$ with each iteration.

 Let us present the steps of the algorithm and then discuss each step in detail.

\noindent
Step 0. Initialize a support matrix $A$.

\noindent
Step 1. For each row of the support matrix perform the Row Equalization Operation.

\noindent
Step 2. Iterate Step 1. Expected number of iterations is (conjecturally) $\log(|E|)$, 
where $|E|$ is the number of hyperedges (rows of $A$).

\noindent
Step 3. Use the resulting support matrix to extract a maximum density subgraph, and more generally, the spectral decomposition of $H$.

\

Now let us give the details.

Step 0. Here we present one reasonable way to initialize the support matrix. 

\noindent
Step 0.1. Set 
$a_{ij} = \left\{ 
\begin{matrix}
1, \ \ {\rm if \ } v_j \in \supp(e_i), \\
0, \ \ {\rm otherwise.} \hfill \\
\end{matrix} \right. $

\noindent
Step 0.2. For each column $j$ of $A$ compute the sum $d_j$ of elements of $A$ in that column (this is going to be the degree of the corresponding vertex). Divide column $j$ by $d_j$.

\noindent
Step 0.3. For each row $i$ of $A$ compute the weighted sum
$v_i = \mathop\sum\limits_{j=1}^n w_j a_{ij}.$
Multiply row $i$ by $u_i / v_i$.

Step 0.3 ensures that we have a proper weighted sum in each row of $A$. Our final goal is to minimize the maximum column sum in $A$. The purpose of Step 0.2 is to avoid the situation where columns corresponding to vertices of high degree create huge column sums.

Step 1. Row Equalization Operation. Consider a row $(a_{i1}, a_{i2}, \ldots, a_{in})$ of matrix $A$. 

Row $i$ of matrix $A$ corresponds to hyperedge $e_i$ in $H$ and by definition of the support matrix, we must have $a_{ij} = 0$ for every $j$ such that $v_j$ does not belong to hyperedge $e_i$. We are going to ignore all such entries. To simplify the notations, we will assume that hyperedge $e_i$ contains vertices $\{ v_1, v_2, \ldots, v_n \}$. Hence in the description of this Step, $n$ is not the total number of vertices in $H$, 
but rather the number of vertices in $e_i$.

Since the Row Equalization Operation deals
with a single row of $A$, within this Step we will denote the entries in this row simply as $(a_1, a_2, \ldots, a_n)$. 
We will also need to know the column sums $(s_1, s_2, \ldots, s_n)$ for matrix $A$.

 The goal of the Row Equalization Operation is to calculate new values for this row $(a_1^\prime, a_2^\prime, \ldots, a_n^\prime)$
in such a way that new column sums $(s_1^\prime, s_2^\prime, \ldots, s_n^\prime)$ possess the following two properties:

(1) There exists a value $\tilde{s}$ such that $s_j^\prime = \tilde{s}$ whenever $a_j^\prime \neq 0$.

(2) If $a_j^\prime = 0$ then $s_j^\prime \geq \tilde{s}$.

Of course, while doing this, we must also preserve the condition
$$\sum\limits_{j=1}^n w_j a_j = u_i. $$

Note that the column sum $s_j$ may be written as $s_j = a_j + b_j$, where $b_j$ is the sum in column $j$ in all rows, except for the current
row. Clearly, the values of $b_j$ will not change as a result of the Row Equalization Operation, hence
$$b_j = s_j - a_j = s_j^\prime - a_j^\prime.$$
In particular, we see from (2) that we shall set $a_j^\prime = 0$ if and only if $b_j \geq \tilde{s}$.
Determining the value of $\tilde{s}$ is the essential part of the Row Equalization Operation.

Let us restate the Row Equalization Problem in the style of grade 3 Math. We need to pour champagne into $n$ glasses of different shapes. We are given the lengths $b_1, \ldots, b_n$ of the stems of these glasses, and the areas of cross sections $w_1, \ldots, w_n$ (we assume that all glasses are cylindrical). We need to fill the glasses with champagne from a bottle of volume $u_i$ in such a way that the level of champagne $\ts$ is the same in all filled glasses, while the glasses with stems longer than $\ts$ remain empty.

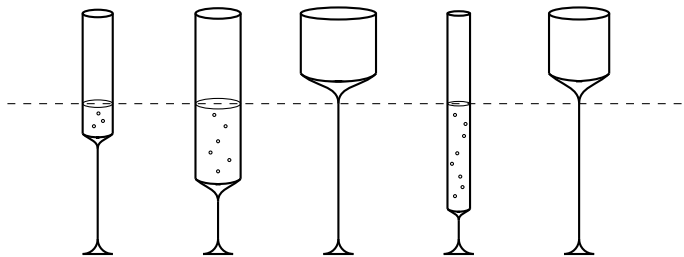
\begin{figure}[h]
\centering
\begin{tikzpicture}
\def\toplev{3.2}
\def\bottomlev{2.4}
% base
\draw[black, thick] (0, 0) -- (0.4 , 0);
\draw[black, thick] (1.6, 0) -- (2.0 , 0);
\draw[black, thick] (3.2, 0) -- (3.6 , 0);
\draw[black, thick] (4.8, 0) -- (5.2 , 0);
\draw[black, thick] (6.4, 0) -- (6.8 , 0);
% stem
\draw[black, thick] (0.2, 0.2) -- (0.2 , 1.4);
\draw[black, thick] (1.8, 0.2) -- (1.8 , 0.7);
\draw[black, thick] (3.4, 0.2) -- (3.4 , \bottomlev - 0.4);
\draw[black, thick] (5.0, 0.2) -- (5.0 , 0.45);
\draw[black, thick] (6.6, 0.2) -- (6.6 , 2.0);
% bottom
\draw[black, thick] (0.0, 1.6) to[out=315, in=0] (0.2, 1.55);
\draw[black, thick] (0.4, 1.6) to[out=225, in=180] (0.2, 1.55);
\draw[black, thick] (1.5, 1.0)  to[out=315, in=0] (1.8, 0.93);
\draw[black, thick] (2.1 , 1.0) to[out=225, in=180] (1.8, 0.93);
%\draw[black, thick] (2.9, 2.2) -- (3.9 , 2.2);
\draw[black, thick] (2.9, \bottomlev) to[out=315, in=0] (3.4, \bottomlev-0.1);
\draw[black, thick] (3.9 , \bottomlev) to[out=225, in=180] (3.4, \bottomlev-0.1);
\draw[black, thick] (4.85, 0.6) to[out=315, in=0] (5.0, 0.56);
\draw[black, thick] (5.15 , 0.6) to[out=225, in=180] (5.0, 0.56);
%\draw[black, thick] (6.2, 2.4) -- (7.0 , 2.4);
\draw[black, thick] (6.2, 2.4) to[out=315, in=0] (6.6, 2.3);
\draw[black, thick] (7.0 , 2.4) to[out=225, in=180] (6.6, 2.3);
% sides
\draw[black, thick] (0.0, 1.6) -- (0.0 , \toplev);
\draw[black, thick] (0.4 , 1.6) -- (0.4 , \toplev);
\draw[black, thick] (1.5, 1.0) -- (1.5, \toplev);
\draw[black, thick] (2.1 , 1.0) -- (2.1 , \toplev);
\draw[black, thick] (2.9, \bottomlev) -- (2.9 , \toplev);
\draw[black, thick] (3.9 , \bottomlev) -- (3.9 , \toplev);
\draw[black, thick] (4.85, 0.6) -- (4.85 , \toplev);
\draw[black, thick] (5.15 , 0.6) -- (5.15 , \toplev);
\draw[black, thick] (6.2, 2.4) -- (6.2 ,\toplev);
\draw[black, thick] (7.0 , 2.4) -- (7.0 , \toplev);
% level
\draw[black, dashed] (-1.0 , 2.0) -- (8.0 , 2.0);
% base curves
\draw[black, thick] (0, 0) to[out=0, in=270] (0.2, 0.2);
\draw[black, thick] (0.4, 0) to[out=180, in=270] (0.2, 0.2);
\draw[black, thick] (1.6, 0) to[out=0, in=270] (1.8, 0.2);
\draw[black, thick] (2.0, 0) to[out=180, in=270] (1.8, 0.2);
\draw[black, thick] (3.2, 0) to[out=0, in=270] (3.4, 0.2);
\draw[black, thick] (3.6, 0) to[out=180, in=270] (3.4, 0.2);
\draw[black, thick] (4.8, 0) to[out=0, in=270] (5.0, 0.2);
\draw[black, thick] (5.2, 0) to[out=180, in=270] (5.0, 0.2);
\draw[black, thick] (6.4, 0) to[out=0, in=270] (6.6, 0.2);
\draw[black, thick] (6.8, 0) to[out=180, in=270] (6.6, 0.2);
%bottom curves
\draw[black, thick] (0, 1.6) to[out=-45, in=90] (0.2, 1.4);
\draw[black, thick] (0.4, 1.6) to[out=225, in=90] (0.2, 1.4);
\draw[black, thick] (1.5, 1.0) to[out=-45, in=90] (1.8, 0.7);
\draw[black, thick] (2.1, 1.0) to[out=225, in=90] (1.8, 0.7);
\draw[black, thick] (2.9, \bottomlev) to[out=-45, in=90] (3.4, \bottomlev - 0.4);
\draw[black, thick] (3.9, \bottomlev) to[out=225, in=90] (3.4, \bottomlev - 0.4);
\draw[black, thick] (4.85, 0.6) to[out=-45, in=90] (5.0, 0.45);
\draw[black, thick] (5.15, 0.6) to[out=225, in=90] (5.0, 0.45);
\draw[black, thick] (6.2, 2.4) to[out=-45, in=90] (6.6, 2.0);
\draw[black, thick] (7.0, 2.4) to[out=225, in=90] (6.6, 2.0);
% top ellipses
\draw[thick] (0.2, \toplev) ellipse (0.2 and 0.05);
\draw[thick] (1.8, \toplev) ellipse (0.3 and 0.07);
\draw[thick] (3.4, \toplev) ellipse (0.5 and 0.08);
\draw[thick] (5.0, \toplev) ellipse (0.15 and 0.03);
\draw[thick] (6.6, \toplev) ellipse (0.4 and 0.08);
% level ellipses
\draw (0.2, 2.0) ellipse (0.2 and 0.05);
\draw (1.8, 2.0) ellipse (0.3 and 0.07);
\draw (5.0, 2.0) ellipse (0.15 and 0.03);
% bubbles
\draw (0.15, 1.7) circle (0.02);
\draw (0.27, 1.77) circle (0.02);
\draw (0.21, 1.87) circle (0.02);
\draw (1.8, 1.1) circle (0.02);
\draw (1.95, 1.25) circle (0.02);
\draw (1.7, 1.35) circle (0.02);
\draw (1.8, 1.5) circle (0.02);
\draw (1.9, 1.7) circle (0.02);
\draw (1.75, 1.85) circle (0.02);
\draw (4.95, 0.77) circle (0.02);
\draw (5.05, 0.89) circle (0.02);
\draw (5.02, 1.03) circle (0.02);
\draw (4.91, 1.2) circle (0.02);
\draw (4.98, 1.34) circle (0.02);
\draw (5.07, 1.57) circle (0.02);
\draw (5.09, 1.73) circle (0.02);
\draw (4.95, 1.85) circle (0.02);
\end{tikzpicture}
% \caption*{Caption}
\end{figure}

\

The value of $\ts$ will be determined via the process of increasing a lower bound for it.
Using the language of the champagne model, we will be selecting a glass and testing whether there is enough champagne in the bottle to fill the glasses to the level of the bottom of the chosen glass. If we succeed, this level will become a new lower bound $b$ for $\ts$. We will keep track of the amount of champagne $u$ already dispensed, and the total cross section area $w$ of current partially filled glasses. We denote by $B$ the set of glasses for which it is currently undetermined whether they are going to be filled or not. Since all glasses with stems lower than $b$ will be filled, all elements in $B$ are greater than $b$.

\noindent
Step 1.0. Initialize with $b \leftarrow 0$, $u \leftarrow 0$, $w \leftarrow 0$,
$B \leftarrow \{ b_1, \ldots, b_n \}$.

\noindent
Step 1.1. Pick an element $b_j \in B$. The volume of champagne required to fill the glasses to level $b_j$ is
$$u^\prime = u + w (b_j - b) + \sum\limits_{\substack{b_k \in B \\ b_k \leq b_j}} w_k (b_j - b_k).$$
Here $u$ is the volume already poured, $w(b_j - b)$ is the additional amount required to raise the level in the partially filled glasses from $b$ to $b_j$, and the last sum represents the volume required to fill empty glasses that have stems shorter than $b_j$ to level $b_j$.

\noindent
Step 1.2. If $u^\prime \leq u_i$ then set
$$ u \leftarrow u^\prime, \quad b \leftarrow b_j, \quad w \leftarrow w + \sum\limits_{\substack{b_k \in B \\ b_k \leq b_j}} w_k$$
and remove from $B$ all elements $b_k$ with $b_k \leq b_j$.

\noindent
Step 1.3. If $u^\prime > u_i$ then remove from $B$ all elements $b_k$ with $b_k \geq b_j$.

\noindent
Step 1.4. Iterate Steps 1.1--1.3 until set $B$ is empty.

\noindent
Step 1.5. Now we know that filled glasses will be precisely those with $b_k \leq b$. The level $\ts$ is now determined by the amount of the remaining champagne:
$$\ts = b + (u_i - u)/w.$$

\noindent
Step 1.6. Set new values in the current row of the support matrix:
$$a_j^\prime = \left\{
\begin{matrix}
\ts - b_j \text{\ if \ } b_j \leq b, \\
0, \text{\ otherwise. } 
\end{matrix}
\right.
$$

We see that computational complexity of Steps 1.1--1.3 is linear in the size of $B$. On average, the size of $B$ decreases exponentially in Steps 1.2/1.3. As a result, computational complexity of the Row Equalization Procedure is linear in $n$ on average, and has complexity of $n^2$ in the worst case scenario (keep in mind that here $n$ is the size of the support of $e_i$ and not the size of $V$). 

 It is possible to give a more complicated algorithm which will have linear complexity in all cases, based on finding a median value in set $B$, which can be done with a linear complexity by applying the median of medians algorithm \cite{B}.

 It may be useful to keep track of the indices in each row corresponding to $b$, as well as the next highest value among $\{ b_1, \ldots, b_n \}$, until the next iteration
of the Row Equalization Operation for the same row, since in many instances these indices will not change and we will speed up our computation if we start the analysis of Step 1.1 with these two indices.

As an illustration, let us present the following example of the Row Equalization Operation. In this example we set all $w_j = 1$ and $u_i = 1$.

\begin{equation*}
\begin{matrix}
j & \hbox{\hskip 0.3cm} 1 \hbox{\hskip 0.3cm}  & \hbox{\hskip 0.3cm}  2 \hbox{\hskip 0.3cm}  &\hbox{\hskip 0.3cm}  3 \hbox{\hskip 0.3cm} & \hbox{\hskip 0.3cm} 4 \hbox{\hskip 0.3cm} & \hbox{\hskip 0.3cm} 5 \hbox{\hskip 0.3cm} & \hbox{\hskip 0.3cm} 6 \hbox{\hskip 0.3cm} \\
a_j & 0.3 & 0.1 & 0.3 & 0.15 & 0.1 & 0.05 \\
s_j & 1.8 & 1.5 & 1.3 & 0.9 & 1.0 & 1.2 \\
b_j & 1.5 & 1.4 & 1.0 & 0.75 & 0.9 & 1.15 \\
% c & 2.3 & 1.8 & 0.35 & 0 & 0.15 & 0.8 \\
a_j^\prime & 0 & 0 & 0.2 & 0.45 & 0.3 & 0.05 \\
s_j^\prime & 1.5 & 1.4 & 1.2 & 1.2 & 1.2 & 1.2 \\
\end{matrix}
\end{equation*}

\

Step 3. Let us describe how we can extract a subgraph from the computed support matrix $A$.

 An optimal matrix will have the property that the maximum column sum $s_{\max} = \alpha(H^\prime)$ will occur in every column corresponding to the vertices of the densest subgraph $H^\prime$. In reality we will only have an approximation to the optimal matrix, so the column sums 
in the columns corresponding to the vertices of $H^\prime$ will only be close to $\alpha(H^\prime)$. For this reason we will be building the set of vertices $V$ and the set of hyperedges $E$ for a subhypergraph using an iterative procedure:

\noindent
 Step 3.0. Start with $V_\alpha = \varnothing$, $E_\alpha = \varnothing$. Determine the column with the maximum column sum in $A$ and add the corresponding vertex to set $V_\alpha$.

\noindent
 Step 3.1. For each new vertex $v_j$ added on the previous step, take the corresponding column of $A$ and for each $i$ with $a_{ij} \neq 0$
add hyperedge $e_i$ into $E_\alpha$ if it is not already there.

\noindent
 Step 3.2. For each new hyperedge $e_i$ added to $E_\alpha$ on the previous step, add all vertices that $e_i$ contains to $V_\alpha$, if these are not already in $V_\alpha$. 

 Iterate Steps 3.1-3.2 until there are no new vertices added to $V_\alpha$.

 It is sufficient to have a good approximation to an optimal support matrix for the optimal subgraph to emerge. We point out that any support matrix gives an upper bound on the density of subgraphs in $H$,
$$s_{\max} (A) \geq \alpha(H^\prime),$$
thus comparing the density of the extracted subgraph to the maximum column sum of the computed support matrix, we can assess how far are we away from the optimal pair. During the iterations of the Algorithm, the maximum column sum $s_{\max}(A)$ will be decreasing, and the density of a subgraph will be increasing until we reach the optimal subgraph.

For unweighted hypergraphs, if we achieve $s_{\max} (A) - \alpha(H^\prime) < 
1/d |V|$, where $d$ is the denominator of a rational number $\alpha(H^\prime)$, this will prove that $H^\prime$ is optimal.

The above algorithm can be easily parallelizable by assigning a subset of rows of $A$ to each processor. The only information that these processors must share is the vector of column sums.

In conclusion, this paper further develops density theory for weighted hypergraphs, connecting maximization of density to spectral and minimax problems for support matrices. Even though the problem we are solving is purely combinatorial, the proposed solution is based on a continuous optimization. We also present an efficient iterative algorithm for finding a maximum density subgraph in a given hypergraph. Each iteration of our algorithm is significantly faster than an iteration of a max-flow based algorithm, thus extending feasibility of solving DSP to even bigger hypergraphs.

\end{document}